\documentclass[12pt]{amsart}
\usepackage{amssymb,latexsym,graphicx,amsfonts,amsthm,amsmath}
\usepackage{enumerate}

\makeatletter
\@namedef{subjclassname@2010}{%
  \textup{2010} Mathematics Subject Classification}
\makeatother

\newtheorem{theorem}{Theorem}[section]
\newtheorem{corollary}[theorem]{Corollary}
\newtheorem{lemma}[theorem]{Lemma}

\newtheorem{proposition}[theorem]{Proposition}

\theoremstyle{definition}

\newtheorem{example}[theorem]{Example}
\newtheorem{question}{Question}

\numberwithin{equation}{section}

      \makeatletter
      \def\@setcopyright{}
      \def\serieslogo@{}
      \makeatother

\def\K{{\bf K}}

\def\frob{\texttt{frob}}

\def\Q{{\bf Q}}
\def\F{{\bf F}}
\def\Xd{{X^{d}_0(N)}}
\def\X{{X_0(N)}}
\def\Xc{{X_0^{\zeta}(N)}}
\def\Z{{\bf Z}}

\def\Pic{{\textrm{Pic}}}

\def\Gal{{\textrm{Gal}}}

\def\Aut{{\operatorname{Aut}}}
\def\End{{\operatorname{End}}}

\begin{document}


\baselineskip=17pt



\title[On Polyquadratic Twists]{On Polyquadratic Twists of $X_0(N)$}

\author[E. Ozman]{Ekin Ozman}
\address{Department of Mathematics\\ University of Texas-Austin\\
Texas, USA}
\email{ozman@math.utexas.edu}


\begin{abstract}
\noindent Let $\K=\Q(\sqrt{d_1},\ldots, \sqrt{d_k})$ be a polyquadratic number field and $N$ be a squarefree positive integer with at least $k$ distinct factors. The Galois group, $\Gal(\K/\Q)$ is an elementary abelian two group generated by $\sigma_i$ such that $g_i(\sqrt{d_i})=-\sqrt{d_i}$. Let $\zeta:\Gal(\K/\Q) \rightarrow \Aut(X_0(N))$ be the cocycle that sends $\sigma_i$ to $w_{m_i}$ where $w_{m_i}$ are the Atkin-Lehner involutions of $\X$. In this paper, we study the $\Q_p$-rational points of the twisted modular curve $\Xc$ and give an algorithm to produce such curves which has $\Q_p$-rational points for all primes $p$. Then we investigate violations of the Hasse Principle for these curves and give an asymptotic for the number of such violations. Finally, we study reasons of such violations.
\end{abstract}



\maketitle

\section{Introduction}

Given $(m_1, \ldots, m_k)$ pairwise relatively prime, squarefree, positive integers and $(d_1,\ldots, d_k)$ relatively prime squarefree integers we construct a twisted modular curve $\Xc$ as follows: Let $N=\Pi_{i=1}^k m_i$ and $\K=\Q(\sqrt{d_1},\ldots, \sqrt{d_k})$. The Galois group of $\K/\Q$ is elementary abelian $2$ group generated by $\sigma_i$ for $1 \leq i \leq k$. The automorphism group of the modular curve $\X$ is generated by the Atkin-Lehner involutions $w_{p_i}$  for each $p_i|N$. Let $\zeta:\Gal(\K/\Q) \rightarrow \Aut(X_0(N))$ be the cocycle that sends $\sigma_i$ to $w_{m_i}$. The curve $\Xc$ is the twist of $\X$ by $\zeta$. In particular rational points of $\Xc$ are $\K$-rational point of $\X$ fixed by $\sigma_i \circ w_{m_i}$ for each $i$ in between $1$ and $k$.

Like $\X$, the twisted curve $\Xc$ is also a moduli space. Rational points of $\Xc$ parametrize $\Q$-curves. Recall that a $\Q$-curve is an elliptic curve $E$ defined over a number field $\K$ which is isogenous to all of its Galois conjugates. It is a result of Elkies that every $\Q$-curve is geometrically isogenous to a $\Q$-curve defined over a polyquadratic field i.e. a field that is generated by quadratic fields. Therefore understanding rational points of $\Xc$ gives information about $\Q$-curves as well. Then one can naturally ask the following question, a particular case of which was first stated in \cite{Ell} and answered in \cite{Ozman} :



\begin{question}\label{Q1}
Given pairwise relatively prime, squarefree, positive integers $(m_1, \ldots, m_k)$, relatively prime squarefree integers $(d_1,\ldots, d_k)$ and a prime number $p$, what can be said about the points in $\Xc(\Q_p)$ where $N=\Pi_{i=1}^k m_i$, $\K=\Q(\sqrt{d_1},\ldots, \sqrt{d_k})$ and $\zeta$ is as described above?
\end{question}

Understanding local points is the first step towards understanding global points of any curve. If it happens to be the case that $\Xc(\Q_p)$ is empty for some $p$, then $\Xc(\Q)$ is also empty and there is no such $\Q$-curve. However, having $\Q_p$-points for every prime $p$ does not guarentee the existence of $\Q$-points unless the genus of the curve is zero. For instance, in the case of quadratic twists there are many examples which have local points, but no global points. It is even possible to give an exact asymptotic formula for the number of such curves \cite{Ozman}. This raises the following question:

\begin{question}  Is there an asymptotic for the number of twists $\Xc$ which violate the Hasse principle?
\end{question}

We address these two questions in the first five sections. In the last section, we discuss further directions and reasons of violations of the Hasse principle. More precisely, in Sections 2 and 3 we give conditions on the existence of local points. These conditions depend on splitting behavior of the prime $p$ in the given polyquadratic field $\K$. In some cases we are able to give necessary and sufficient conditions, in other cases we have only sufficient ones. However, this still allows us to give an algorithm which produces infinitely many $\Xc$ with local points everywhere, for any given $N$, as summarized in Section 4. In Section 5, we use this algorithm combined with the methods of \cite{Ozman} and \cite{Clark}, and give an asymptotic formula for the number of biquadratic twists which has local points everywhere but no global points. In fact this can be generalized to higher degree twists as well. Section 6 is about further directions in this problem. For the twists with computationally feasible equations, we try to explain the lack of global points using Mordell Weil sieve. This is equivalent to Brauer Manin obstruction by the work of Scharaschkin \cite{sch}. One can not apply Mordell Weil sieve if $\Pic^1(\Xc)(\Q)$ is empty. Understanding the Picard group is usually hard for a generic curve. Note that we even do not have equations for an arbitrary member of the twisted family. Given $N$, we give sufficient conditions for $\Pic^1(\Xc)(\Q)$ to be nonempty in the case of quadratic and some biquadratic cases. We also find families of cases where $\Pic^1(\Xc)(\Q_p)=\emptyset$ when $p$ and $N$ satisfy certain arithmetic conditions.

During the course of typing these results, the PhD thesis of Jim Stankewicz has been brought to author's attention. Many of the results in Sections 2 and 3 can be concluded from this thesis. However, we still included them here since it may be hard to derive these results from \cite{Jim} for a reader who is not familiar with the subject and also the work was independent. 




\section{The case of good reduction}

\subsection{$p$ is unramified in $K$}: In this case the extension is unramified therefore we can use the theory of Galois descent. Since $p$ is a good prime, $\Xc$ has a smooth model over $\Z_p$. Therefore, by Hensel's Lemma, $\Xc(\Q_p)$ is non-empty if and only if $\Xc(\F_p)$
is non-empty. Let $\mathcal{P}$ be a prime of $\K$ lying over $p$. According to our notation given in the introduction section, each Galois map $\sigma_i$ is twisted by some $w_{m_i}$. Let $S$ be the set of indices $i$ such that $p$ is inert in $\Q(\sqrt{d_i})$. Then decomposition group of $\mathcal{P}$ is $\{1, \prod\limits_{i \in S} \sigma_i \}$. Note that since $\K/\Q$ is an abelian Galois extension, it doesn't matter which $\mathcal{P}$ we choose. The map $\prod\limits_{i \in S} \sigma_i $ induce frobenius map on the level of residue fields.
 Note that the coycle $\zeta$ twists the action of $\prod\limits_{i \in S} \sigma_i $ by $\prod\limits_{i \in S} w_{m_i}$. Therefore $\Xc(\F_p)$ consists of $\F_{p^2}$-rational points of $\X$ that are fixed by $w_M \circ \frob$ where $M=\prod\limits_{i \in S}m_i$.

Note that, if $S$ is empty then $p$ splits completely in $\K/\Q$ and $\K \hookrightarrow \K_{\mathcal{P}} \cong \Q_p$. Therefore, $\Xc(\Q_p)=\X(\Q_p)\neq \emptyset$.

The other extreme case is when $p$ is inert in each $\Q(\sqrt{d_i})$, i.e. $S=\{1,2,\ldots,k\}$. We will show that in this case there are points in $\Xc(\F_p)$. Our strategy is to prove that there is a 
supersingular point fixed by $w_N \circ \frob$, or equivalently, by $w_N \circ w_p =w_{Np}$. This will be derived from well-known results in quaternion arithmetic, see \cite{Vig} page 152. Using more advanced tools of quaternion arithmetic we can give sufficient conditions for the existence of a $\Q_p$-point on $\Xc$ when $0<|S|<k$.

\begin{proposition}\label{unramifiedgood}
Let $p$ be a prime which doesn't divide $N$ and unramified in $K$. Let $M=\prod\limits_{i \in S} m_i$ then if $\left ( \frac{-pM}{p_i} \right )=1$ for all $p_j |N$ and $j \notin S$ then $\Xc(\Q_p)$ is nonempty.

\end{proposition}

Remark: If $S=\emptyset$ or $S=\{1,\ldots,k\}$ then Proposition \ref{unramifiedgood} implies that $\Xc(\Q_p)$ is non empty. In other words, if $p$ splits in each $\Q(\sqrt{d_i})$ or inert in each $\Q(\sqrt{d_i})$ then $\Xc(\Q_p)$ is non-empty.

\begin{proof}
The idea will be similar to the idea of Theorem 3.17 in \cite{Ozman}. We will show that there is a supersingular point in $\Xc(\F_p)$. For the convenience of the reader, we produce the related parts of the proof here again.

Let $\Sigma_N$ be the set of tuples $(E,C)$ such that $E$ is a supersingular elliptic curve over characteristic $p$ and $C$ is cyclic group of order $N$. We start by studying the action of the involution $w_N \circ \sigma$ on $\Sigma_N$. 
Let $B:=\End(E) \otimes_{\Z} \Q$, then $B$ is the unique quaternion algebra over $\Q$ which is ramified only at $p$ and $\infty$, $\End(E)$ is a maximal order in $B$, and $\End(E,C)$ is an Eichler order of level $N$. The frobenius map acts on the set of singular points of $X_0(Np)_{/\F_p}$ as $w_p$ (see Chapter V, Section 1 of \cite{Deligne} or Proposition 3.8 in \cite{Ribet2}).


The modular curve $X_0(Np)$ has bad reduction over $\F_p$. A regular model of $X_0(N)_{/ \F_p}$ is given by Deligne-Rapaport and consists of two copies of $X_0(N)_{/ \F_p}$ glued along supersingular points.  Let $\psi$ be a map from $X_0(N)_{/ \F_p}$ to $X_0(Np)_{/ \F_p}$, an isomorphism onto one of the two components. The map $\psi$ takes the supersingular locus of $X_0(N)_{/ \F_p}$ to the supersingular locus of $X_0(Np)_{/ \F_p}$. The set $\Sigma_N$ is the supersingular locus of $X_0(N)_{/ \F_p}$. The Atkin-Lehner operators $w_{Mp}$ acts on $X_0(Np)$ for every $M|N$, in particular $w_{Mp}$ acts on $\psi(\Sigma_N)$. When we say the action of $w_{Mp}$ on $\Sigma_N$, it is meant the action of $w_{Mp}$ on $\psi(\Sigma_N)$.

The involution $w_{Mp}$ has a fixed point on $\Sigma_N$ if and only if $\Z[\sqrt{-Mp}]$ embeds in $\End(E,C)$ for some $(E,C) \in \Sigma_N$. Since $\End(E,C)$ is an Eichler order of level $N$ in $\Q_{p,\infty}$ by Eichler's optimal embedding theorem(see \cite{Vig}) $\Z[\sqrt{-Mp}] \hookrightarrow \End(E,C)$ if and only if we have the condition given in the proposition statement. This implies that there is a supersingular point in $\Xc(\F_p)$ and by Hensel's lemma there is a point in $\Xc(\Q_p)$.

\end{proof}

By Proposition \ref{prop:unramifiedgood} we have a sufficient condition for $\Xc(\Q_p) \neq \emptyset$ and if we have the conditions given in this proposition then reduction of any point in $\Xc(\Q_p)$ has to be a supersingular point. If the condition of the Proposition \ref{prop:unramifiedgood} fails, there may still be ordinary points in $\Xc(\F_p)$ which are points in $\X(\F_{p^2})$ fixed by $\frob \circ w_M$. 

Now we will assume that the condition of the Proposition \ref{prop:unramifiedgood} fails i.e. $\left ( \frac{-pM}{p_j} \right )=-1$ for some $p_j|N$ and $j \notin S$. In this case if there is $x$ in $\Xc(\F_p)$ then $x$ has to be an ordinary point. The following proposition gives the conditions for the existence of a $\F_p$-rational, $w_M$ fixed ordinary point of $\X$ which is another sufficient condition for $\Xc(\F_p)\neq \emptyset$. Note that if both Propositions \ref{prop:unramifiedgood} and \ref{ordinarycm} fail then $\Xc(\F_p)$ can still be non-empty but such a point has to be an ordinary point defined over $\F_{p^2}-\F_p$ and fixed by $w_M \circ \frob$.

\begin{lemma}\label{lemmacmlifting} Let $p$ be an odd prime and $M,p_j$ be as above. There is a $w_M$-fixed $\Q_p$-rational point on $X_0(N)$ if and only if there is a $w_M$-fixed $\F_p$-rational point on $X_0(N)_{/\F_p}$.
\end{lemma}

\begin{proof} 
Say $(E,C^{N/M} \oplus C^M)$ is a $w_M$-fixed point of $X_0(N)(\F_p)$, where $C^{N/M}, C^M$ denotes cyclic subgroups of order $N/M,M$. In particular $C^{N/M}$ can be written as $\oplus C^{p_j}$ where $j \notin S$. Then $\ker(\lambda)=C^M$ and $\lambda(C^{p_j})=C^{p_j}$  for some endomorphism $\lambda$ of $E$. In particular, $p_j$ divides separable degree of $\lambda-[x]$ where $[x]$ denotes multiplication by $x$ map for some $1\leq x \leq p_j-1$. We have a short exact sequence as follows:

$$ 0 \rightarrow C^M \rightarrow E \stackrel{\lambda}{\rightarrow} E \rightarrow 0$$

By Deuring's Lifting Theorem(\cite{Deuring}), this short exact sequence can be lifted to a short exact sequence as below where everything is defined over a number field $B$ and there is a prime $\nu$ of $B$ with residue degree $1$ and the reduction of $\tilde{E}, \tilde{\lambda}$ mod $\nu$ gives us the sequence above. 

$$ 0 \rightarrow \tilde{C^M} \rightarrow \tilde{E} \stackrel{\tilde{\lambda}}{\rightarrow} \tilde{E} \rightarrow 0.$$

The lifted curve $\tilde{E}$ and the lifted map $\tilde{\lambda}$ give rise to a $w_M$-fixed point on $\X(B)$ if and only if there exists cyclic groups of order $p_j$ in $\ker(\tilde{\lambda}-[x])$  for all $j$, where $[x]$ is multiplication by $x$ map for some integer $x$ in $\{1,\ldots,p_j-1\}$. By assumption, $p_j$ divides the degree of $\lambda-[x]$ for some $x$ in $\{1,\ldots,p_j-1\}$. Since $(p,M)=1$, degree of $\lambda=M$ is the same as degree of $\tilde{\lambda}$. Therefore $\tilde{E}$ gives rise to a $w_M$ fixed point in $\X(B)$ since both inertia and residual degrees of $p$ in $B$ are one, $B \hookrightarrow \Q_p$.

Conversely, since the fixed locus of $w_M$ is proper, a $w_M$-fixed point on $X_0(N)(\Q_p)$ reduces to a $w_M$-fixed point on $X_0(N)(\F_p)$.
\end{proof}

Remark: For $p=2$, the argument goes as in \cite{Ozman} Section 4. 

\begin{proposition}\label{ordinarycm}
Let $p$ be an odd prime. There is a $w_M$-fixed point in $X_0(N)(\Q_p)$ whose reduction to $\F_p$ is ordinary if and only if $\left( \frac{-M}{p} \right)  = 1$ and the number field $\Q(H(-M))$ has a prime lying over $p$ with residue degree one where $H(-M)$ denotes the Hilbert class polynomial of reduced discriminant $M$. If $p=2$, since the unique ordinary elliptic curve over $\F_2$ has endomorphism ring $\Z[\frac{1+\sqrt{-7}}{2}]$, $M=7$.
\end{proposition}

Note that since $\Xc$ has good reduction at $p$, the big primes(compared to the genus of $\Xc$) are not problematic, i.e. $\Xc(\F_p)\neq \emptyset$ for all such primes. Hence we have the following result for $p$ unramified in $K$ and not dividing $N$.

\begin{proposition} \label{prop:unramifiedgood}
Let $p$ be a prime not dividing $N$ and inert in the quadratic fields $\Q(\sqrt{d_i})$ for $i \in S$. Let $M$ be the product of $m_i$ for $i \in S$. Then $\Xc(\Q_p)$ is non-empty if one of the following holds:
\begin{itemize}
\item $p > 4g^2$ where $g$ is the genus of $\X$. 
\item $S$ is empty or $S=\{1,\ldots, k\}$.
\item $\left(\frac{-pM}{p_j}\right )=1$ for $p_j|N$ and $j \notin S$.
\item $\left(\frac{-pM}{p_j}\right )=-1$ for some $p_j|N$ and $j \notin S$:  Hilbert class polynomial of reduced discriminant $M$ has a root mod $p$ if $p$ is odd and $M=7$ if $p=2$.

\end{itemize}
\end{proposition}

\subsection {$p$ is ramified in $\K$}: Let $p$ be ramified in only one of the fields $\Q(\sqrt{d_i})$ and splits in the others. Let $\nu$ be a prime of $\K$ lying over $p$. In this case we don't have the theory of Galois descent since the extension is ramified. The $\Q_p$-rational points of  $\Xc$ are points in $\X(K_\nu)$ that are fixed by $\sigma_{d_i} \circ w_{m_i}$. This case is quite similar to the case of quadratic twists, see \cite{Ozman}.


By Lemma \ref{lemmacmlifting}, we have the necessary and sufficient conditions to lift a $w_{m_i}$ fixed point of $\X(\F_p)$ to a $w_{m_i}$-fixed point of $\X(\Q_p)$, which gives a point in $\Xc(\Q_p)$. The following proposition gives us the converse. The proof is very similar to the proof of Proposition 4.4 in \cite{Ozman} so we don't reproduce the proof here.

 \begin{proposition}\label{prop:diag}
   Let $x$ be a point of $X_0(N)(\K_{\nu})$ such that $w_{m_i}(x^{\sigma_{d_i}})=x$ then $x$ reduces to a $w_{m_i}$-fixed point on the special fiber of $\mathcal{X}_0(N)_{/R}$ where $R$ is the integer ring of $\K_{\nu}$.
 \end{proposition}

Combining Lemma \ref{lemmacmlifting} and Proposition \ref{prop:diag} we can conclude the following:

\begin{proposition}\label{prop:inertram}
Let $p$ be an odd prime that doesn't divide $N$ and ramified only in one of the quadratic fields, namely $\Q(\sqrt{d_i})$ and splits in the others. Then $\Xc(\Q_p)$ is non-empty if and only if there is a prime $\nu$ in $\Q(H(-M))$ lying over $p$ with inertia degree one. 
 \end{proposition}

If $p=2$, by Lemma 4.9 \cite{Ozman}, any $w_M$-fixed point of $\X(F_2)$ can be lifted to an elliptic curve $\tilde{E}$ over a number field $B$ such that $\tilde{E}$ has complex multiplication by the maximal order of $\Q(\sqrt{-M})$. Hence $2$ is unramified in $B$ and $ B \hookrightarrow \Q_2$, there is a $w_M$-fixed point in $\X(Q_2)$.

\section{Bad Primes}

In this section we will deal with the primes diving $N$. We'll assume that these primes are unramified in $\K$. Fix a prime divisor $p=p_{i_0}$ of $N$. As in the case of good primes, if $p$ splits totally in $\K$ then $\K$ embeds in $\Q_p$, hence $\Xc(\Q_p)=\X(\Q_p)\neq \emptyset$. The results in other cases can be summarized as below:


\begin{theorem} \label{thm:mainbad} Let $N=p_1\ldots p_k$ and $p=p_{i_0}$ be an odd prime dividing $N$. Let $S$ be set of indices$i$ such that $p$ is inert in every $\Q(\sqrt{d_i})$ and splits in the rest. 
\begin{enumerate}
\item If $i_0$ is in $S$ (i.e. $p$ is inert in the quadratic field twisting the Galois action by $w_p$) then:
If $p$ is odd, $\Xc(\Q_p) \neq \emptyset$ if and only if $N=p\prod q_i$ or $N=2p\prod q_i$ where $p \equiv 3 \mod 4, q_i \equiv 1 \mod 4$ and $S$ consists of only $i_0$. If $p=2$, $\Xc(\Q_p) \neq \emptyset$ if $N=2\prod q_i$ where $q_i \equiv 1 \mod 4$ and $S$ contains all $i$. (i.e. $2$ is inert in all the quadratic number fields $\Q(\sqrt{d_i})$.) or $S$ contains only $i_0$ ($2$ is inert in only one of the quadratic fields.)
\item If $i_0$ is not in $S$ (i.e. $p$ spits in the quadratic field twisting the Galois action by $w_p$) then:
If $p$ is odd and $S$ contains all $i$ such that $p_i$ is an odd prime dividing $N$ and $i \neq i_0$ i.e. $N=2pM$ or $N=pM$ where $M$ is the product of $q_i$ for $i \in S$ and $q_i \equiv 1 \mod 4$ then $\Xc(\Q_p) \neq \emptyset$. If $p=2$ and $S$ contains all $i$ such that $p_i$ is an odd prime dividing $N$ i.e. $N=pM$ where $M$ is the product of $q_i$ for $i \in S$ and $q_i \equiv 1 \mod 4$ then $\Xc(\Q_p) \neq \emptyset$.
\end{enumerate}
\end{theorem}

In this case we are restricting ourselves to primes unramified in each $\Q(\sqrt{d_i})$. Therefore, we have the theory of Galois descent, we can use Hensel's Lemma. Note that since we are in the case of bad reduction, the first thing we need is a regular model for $\Xc_{/ \F_p}$ which is the twist of the regular model of $\X_{/\F_p}$. Deligne-Rapaport gives a regular model of $\X_{/\F_p}$. This model consists of two copies of $X_0(N/p)_{/\F_p}$ glued along supersingular points and becomes regular after blowing up $(\frac{|\Aut(x)|}{2}-1)$-many times at each intersection point. For instance if every intersection point has automorphism group $\{\pm 1\}$ then the model is already regular. Using this model and Hensel's Lemma we can conclude that there is a $\Q_p$-point on $\Xc$ if and only if there is a smooth point in $\Xc(\F_p)$.

Let $S$ be the set of indices $i$ such that $p=p_{i_0}$ is inert in $\Q(\sqrt{d_i})$. We will study the cases $i_0 \in S$ and $i_0 \notin S$ separately. 

\subsection {$i_0 \in S$}:  Then the decomposition group of $p$ in $\K$ is $\{1, \prod\limits_{i \in S} \sigma_{d_i}\}$ and an $\F_p$-rational of $\Xc$ is a point in $\X(\F_{p^2})$ fixed by $ \prod\limits_{i \in S} w_{p_i} \circ \frob=w_M \circ \frob$ where $M=\prod \limits_{i \in S} p_i$. Note that $p|M$. Since $w_p$ interchanges each branch, $X_0(N/p)$ and frobenius acts on each branch, a $w_M \circ \frob$-fixed point has to be an intersection point i.e. a supersingular point. Moreover frobenius acts as $w_p$ on the set of supersingular points(as mentioned in the section of good reduction), hence $\Xc(\F_p)$ is nonempty if and only is there is a smooth supersingular point in $X_0(N/p)(\F_{p^2})$ fixed by $w_p \circ w_M= w_{M/p}$.

\begin{proposition}\label{prop:degree}
Using the notation above, there is a $w_{M/p}$ fixed smooth supersingular point in $X_0(N/p)(\F_{p^2})$ if and only if there is a Eichler order $O$ of level $N/p$ in the quaternion algebra ramified at $p$ and infinity such that both $\Z[\sqrt{-M/p}]$ and $\Z[i]$ embed simultaneously into $O$.
\end{proposition}

\begin{proof}
Let $x$ be a supersingular $w_{M/p}$-fixed point on $X_0(N)(\F_{p^2})$. This is equivalent to say that $\Z[\sqrt{-M/p}]$ embeds in $\End(x)$, which is an Eichler order of level $N/p$ in the quaternion algebra ramified at $p$. We will show that in order to be smooth $\Z[i]$ must embed in $\End(x)$ as well.

Recall that at each singular(hence supersingular) point $x$ we have $(|\Aut(x)|/2-1)$-many exceptional lines. The automorphism group of an elliptic curve over a field of characteristic $q$ is $\mu_2,\mu_4$ or $\mu_6$ if $q$ is not $2$ or $3$ where $\mu_s$  denotes the group of primitive $s$-th roots of unity. If $q=2$ or $3$ and $E$ is the unique supersingular elliptic curve in characteristic $q$ then $\Aut(E)$ is $C_3 \rtimes \{\pm1,\pm i,\pm j,\pm k\}$ or $C_3 \rtimes C_4 $ respectively, where $C_m$ denotes the cyclic group of order $m$. 

Therefore if $|\Aut(x)|=4n$ for $n>1$, there is an element of order $4$ in $\Aut(x)$ and the number of blow-ups is $2n-1$ which is odd. Since we have odd number of exceptional lines, there is one line $L_{/ \F_p}$ that is fixed by the Galois action. On this line the there are $p+1$ rational points, two of which are singular. 
Therefore if $\Z[i]$ embeds in $\End(x)$ then $x$ is a smooth point. 

Conversely any smooth point is of this form. If $|\Aut(x)|$ is $2$, then the model of $X_0(N)/\Z_p$ is already regular hence no need to blow-up, $x$ is singular. If $|\Aut(x)|=6$, then we replace this point by 2 exceptional lines over $\F_p$ and $\frob \circ w_N$ interchanges these lines. Each of these exceptional lines cuts one of the branches and also the other exceptional line once. Denote the intersection point of these lines by $x'$. Then $x'$ induces an $\F_p$-rational point of $\Xc$. However, it is a singular point. 

Hence $x$ is a smooth point on $\Xc(\F_p)$ if and only if both $\Z[\sqrt{-M/p}]$ and $\Z[i]$ embed in $\End(x)$. 
\end{proof}

 There is an Eichler order $O$ of level $N/p$ in the quaternion algebra $\Q_{p,\infty}$  such that $\Z[i]$ embeds in $O$ if and only if $N=p\prod q_i$ or $N=2p\prod q_i$ where $p \equiv 3 \mod 4$ and $q_i \equiv 1 \mod 4$. Another order $\Z[\sqrt{-M/p}]$ also embeds into $O$ whenever $O$ has an element $a$ of norm $-M/p$. If we consider the $\Z$-module generated by $1,i,a,ia$ we get relations between discriminant of $O$ and the $\Z$-module. Using this idea and properties of quaternion arithmetic we can deduce the following result, a nice proof of which is also given in \cite{Jim}.
 
 \begin{proposition}\label{prop:jim}\emph{(\cite{Jim} Corollary 4.2.3)}
 Let $B$ be a definite quaternion algebra ramified at $D'$  and let $O$ be an Eichler order of $B$ of squarefree level $N'$ such that $\Z[i] \hookrightarrow O$. If $m | D'N'$ and $m \neq 1$ then $\Z[\sqrt{-m}] \hookrightarrow O$ if and only if $m=D'N'$ or $2|D'N'$ and $m=D'N'/2$.
 \end{proposition}

Combining Proposition \ref{prop:jim}, Proposition \ref{prop:degree} and the observation following Proposition \ref{prop:degree} we obtain the first part of Theorem \ref{thm:mainbad}.

\subsection{$i_0 \notin S$} Then the decomposition group of $p=p_{i_0}$ in $\K$ is $\{1, \prod\limits_{i \in S} \sigma_{d_i}\}$ and an $\F_p$-rational of $\Xc$ is a point in $\X(\F_{p^2})$ fixed by $ \prod\limits_{i \in S} w_{p_i} \circ \frob=w_M \circ \frob$ where $M=\prod \limits_{i \in S} p_i$. Note that $p \nmid M$. Note that since $w_{p}$ is the only involution that interchanges the branches, $\frob \circ w_M$ doesn't interchange the branches. Hence an $\frob \circ w_M$-fixed point need not be a supersingular point. This is the main difference with the case $i_0 \in S$. However, as the proposition below shows, we can still find a smooth $\frob \circ w_M$-fixed point among the supersingular points in some cases.

\begin{proposition}\label{onesplitsuper}
Using the notation above, $\Xc(\Q_p)$ is nonempty if we are in one of the following cases:
\begin{itemize}
\item $p \equiv 3 \mod 4$, $N=p\prod q_i$ or $N=2p\prod q_i$ with $q_i \equiv 1 \mod 4$  and $M=\prod q_i$
\item $p=2$ and $N=2\prod q_i$ with $q_i \equiv 1 \mod 4$  and $M=\prod q_i$
\end{itemize}
\end{proposition}

\begin{proof}
As noted above there exists a smooth supersingular point $x$ in $\Xc(\F_{p})$ if and only if $\Z[i]$ embeds in $\End(x)$ and $x$ is fixed by $\frob \circ w_M=w_{Mp}$ since frobenius acts on $x$ as $w_{p}$. This is equivalent to embed both $\Z[i]$ and $\Z[\sqrt{-Mp}]$ into $\End(x)$ which happens if and only if we have the conditions above as cited in Proposition \ref{prop:jim}. 
\end{proof}

Remark: Note that Proposition \ref{onesplitsuper} is not an if and only if statement. If the conditions of the theorem fails, this only tells that there aren't any smooth supersingular point in $\Xc(\F_{p})$. But there may still be ordinary points(which are necessarily smooth) which can be lifted to $\Xc(\Q_{p}$. For instance, by Proposition \ref{ordinarycm}, we have the necessary and sufficient conditions to lift a $w_{p}$ fixed point of $\X(\F_p)$ to a $w_{p}$-fixed point of $\X(\Q_p)$, which gives a point in $\Xc(Q_p)$.

Remark: This corrects conditions given in Theorem 3.7 in \cite{Ozman}.

\section{Algorithm}

In this section we give an algorithm (using the results of previous sections) which produces $\Xc$ with $\Q_p$ points for every prime $p$. More precisely the input of the algorithm is:

 Pairwise relatively prime squarefree positive integers $(m_1,\ldots,m_k)$ and its output is polyquadratic fields $\K=\Q(\sqrt{d_1},\ldots,\sqrt{d_k})$ such that $\Xc(\Q_p)\neq \emptyset$ for all primes $p$ where $N=\prod\limits_{i=1}^k m_i$ and $\zeta:G_{\Q} \rightarrow \Aut(\X)$ such that $\sigma \mapsto \sigma \circ w_{m_i}$ if $\sigma(\sqrt{d_i})=-\sqrt{d_i}$ and $1$ otherwise. Note that $\sigma_i$ is the Galois map that sends $\sqrt{d_i}$ to $-\sqrt{d_i}$.

Given $N=\prod\limits_{i=1}^k m_i$ as above proceed as below:

\begin{enumerate}
\item Choose $d_i$ such that there is no prime simultaneously ramified in $\Q(\sqrt{d_i})$ and $\Q(\sqrt{d_j})$ for $i \neq j$.
\item For all $p|N$ choose $d_i$ such that $p$ splits in each $\Q(\sqrt{d_i})$. Since there are finitely many such $p$, it is possible to choose $d_i$ accordingly.
\item For all $p \nmid N$ and $p < 4g^2$ where $g$ is the genus of $X_0(N)$, choose $(d_1,\ldots,d_k)$ such that $p$ is inert in $\Q(\sqrt{d_i})$ for all $i$ in $[1, r]$ or $p$ splits completely in $\K$. 

\item  For all $p \nmid N$, $p > 4g^2$ and $p$ ramified in one(and only one by first step) of the $\Q(\sqrt{d_i})$ the Hilbert class polynomial of reduced discriminant $-m_i$ has a root mod $p$.
\end{enumerate}

\section{Density Results}

In this section we give an asymptotic for the number of bi-quadratic twists $\Xc$ that has local points everywhere but no global points. To make things more concrete we will make the following assumptions on $N$ but similar asymptotic can be find for other $N$ as well.

\begin{itemize}
\item $m_1,m_2,d_1$ be primes congruent to $1$ mod $4$ and $m_1 \equiv m2 \mod 8$ such that $\left( \frac{d_1}{m_i} \right ) =1$ and $\left( \frac{-2m_1}{m_2} \right)=1$ 

\item Hilbert class polynomial of discriminant $-4m_1$ has a root modulo $d_1$

\item Every prime $p<4g^2$ splits in $\Q(\sqrt{d_1})$ where $g$ is the genus of $\X$. 
\end{itemize}

The curve $\Xc$ is the twist of $\X$ via the cocyle $\zeta$ which sends $\sigma$ to $w_{m_1}$ if $\sigma(\sqrt{d_1})=-\sqrt{d_1}$, to $w_{m_2}$ if $\sigma(\sqrt{d_2})=-\sqrt{d_2}$ and trivial otherwise.

Given a positive integer $X$, let $A'$ be the set of positive squarefree integers $d_2\leq X$ such that $\Xc(\Q_p)$ is non-empty for all $p$ when there is no prime $p$ simultaneously ramified in $\Q(\sqrt{d_1},\sqrt{d_2})$ and $\Q(\sqrt{-N})$.

\begin{proposition}\label{prop:densityodd}

Keeping the notation as above, we have that 

$$|A'|=M_{S_N} \frac{X}{\log^{1-\alpha}X}+ O\left (\frac{X}{\log^{2-\alpha}X} \right )$$ where $\alpha$ is the density of the set of primes $p$ such that Hilbert class polynomial of reduced discriminant $-m_2$ has a root mod $p$ and $\left( \frac{d_1}{p} \right)=1$.

\end{proposition}

In order to prove this proposition we need the following result of Serre:

\begin{theorem}[Serre, Theorem 2.8 in \cite{serreden}] \label{thm:serreden}
Let $0<\alpha <1$ be  Frobenius density of a set of primes $S$ and $N_S(X)$ is the number of squarefree integers in $[1 \ldots X]$ all of whose prime factors lie in $S$. Then $N_S(X) = c_S\frac{X}{\log^{1-\alpha}X}+ O\left (\frac{X}{\log^{2-\alpha}X} \right )$ for some positive constant $c_S$. 
\end{theorem}

\emph{Proof of proposition:} Let $S_N$ be the set of primes $p$ such that Hilbert class polynomial of reduced discriminant $-m_2$ has a root mod $p$ and $\left( \frac{d_1}{p}\right)=1$. Since $S_N$ is a Chebotarev set of primes, it has a well defined density $\alpha$. Let $A=\{d \in \Z | d \leq X,\textrm{squarefree},(d,N)=1, d=\Pi_i p_i, p_i \in S_N\}$ then the density of $A$ is given by the above theorem of Serre. 

Let $p=2$. Since $2$ is ramified in $\Q(\sqrt{-N})$, $2$ must be unramified in $\Q(\sqrt{d_1},\sqrt{d_2})$. Therefore we should consider $d_2$ in $A$ such that $d_2 \equiv 1$ mod $4$. Since $2 \nmid N$, $2$ is an unramified good prime. By Proposition \ref{prop:inertram} if $2$ splits in both $Q(\sqrt{d_i})$ or inert in both, then $\Xc(N)(\Q_2)$ is nonempty. In the case that $2$ is inert in one of the quadratic fields and split in the other one, the conditions $m_i \equiv 1$ mod $4$, $m_1 \equiv m2 \mod 8$ and $\left( \frac{-2m_1}{m_2} \right)=1$ guarantees the existence of a $\Q_2$-point. 

Now we will show that for each $d_2$ in $A$ such that $d_2 \equiv 1$ mod $4$, $\Xc(\Q_p)$ is nonempty. We begin with $p \nmid N$. If $p$ is inert in both $\Q(\sqrt{d_1})$ and $\Q(\sqrt{d_2})$ or splits in both then $\Xc(\Q_p)$ is nonempty by Proposition \ref{prop:inertram}. 

Let $p$ be inert in $\Q(\sqrt{d_1})$ and splits in $\Q(\sqrt{d_2})$. Then by assumption $p$ has to be greater than $4g^2$. Hence using Weil bounds and Hensel's Lemma, $\Xc(\Q_p)$ is nonempty. 

Say $p$ is inert in $\Q(\sqrt{d_2})$ and splits in $\Q(\sqrt{d_1})$. If $p>4g^2$, there are local points, so say $p<4g^2$. If $\left( \frac{-pm_2}{p_i} \right)=1$ for all $p_i | m_1$, there are local points by Part c of Proposition \ref{prop:inertram}. On the other hand if $\left( \frac{-pm_2}{p_i} \right)=-1$ for some $p_i |m_1$, then we are done with last part of the same proposition since $p \in S_N$ and $p$ is an odd prime.

Let $p$ be ramified in $\Q(\sqrt{d_1})$. In particular $p=d_1$. Since $p \in S_N$, $\left( \frac{d_1}{p}\right)=1=\left( \frac{p}{d_1}\right)$, hence $p$ splits in $\Q(\sqrt{d_2})$. Then with the given conditions on $m_1,m_2$ and $d_1$ in the beginning, $\Xc(\Q_p)$ is nonempty by Proposition \ref{prop:unramifiedgood}. Similarly, let $p$ be ramified in $\Q(\sqrt{d_2})$. Then $p|d_2$, $p\in S_N$, hence $\left( \frac{d_1}{p}\right)=1$, $p$ splits in $\Q(\sqrt{d_1})$ hence $\Xc(\Q_p)$ is nonempty by the same proposition.
 
For $p=m_1$ or $p=m_2$ we show that $p$ splits in $\Q(\sqrt{d_1}, \sqrt{d_2})$. By given conditions, $p$ splits in $\Q(\sqrt{d_1})$. The genus field of $\Q(\sqrt{-N})$ is $\Q(\sqrt{-1},\sqrt{m_1},\sqrt{m_2})$ and ring class field of $\Z[\sqrt{-N}]$ is $\Q(\sqrt{-N},j(\sqrt{-N}))$ and $j(\sqrt{-N})$ is real(\cite{Cox}). Therefore $\Q(\sqrt{m_1},\sqrt{m_2})$ lies in $\Q(j(\sqrt{-N}))$. Let $q$ be a prime divisor of of $d$. Note that $q$ has to be odd. Since $q$ is in $S_N$, there is a prime $\mathcal{Q}$ of $\Q(j(\sqrt{-N}))$ lying over $q$ with inertia degree $1$. Therefore, $\left ( \frac{p}{q} \right )=\left ( \frac{q}{p} \right )=1$, hence $p$ splits in $\Q(\sqrt{d_2})$. 

Therefore under the given conditions on $m_i$ and $d_1$, for any $d_2$ in $A$ and $d \equiv 1$ mod $4$, $\Xc(N)(\Q_p) \neq \emptyset$ for all $p$. This gives us the claimed asymptotic.

\vspace{0.5cm}
Using this result, one can write down a curve $X_0(N)$ and compute an explicit asymptotic for the set of bi-quadratic twists of $\Xc$ violating the Hasse principle using Faltings' finiteness results as done by Clark in the proof of Theorem 2 in \cite{Clark}. The following is essentially same as Theorem 2 in \cite{Clark}.

\begin{theorem}\label{fintwist}
Let $D,N$ be a squarefree integers and $m$ be a divisor of greater than $427$. Let $K=\Q(\sqrt{D})$ and $L$ be a quadratic extension of $K$. Consider the cocycle $\zeta':\sigma \mapsto w_m$ where $\sigma$ is the generator of the Galois group of $L/K$. Let $X^D/K$ denote the twist of $\X/K$ via $\zeta'$. Hence $X^D(K)=\{P \in \X(L)| \sigma(P)=w_m(P) \}$. Then there are only finitely many $D$ such that $X^D(K)$ is nonempty. 
\end{theorem}

\begin{proof}
Consider the following maps:

$$\alpha_D: X^D(N)(K) \hookrightarrow \X(L)$$ and 

$$\beta_D: \X(L) \rightarrow \X/w_m (L)$$

Let $C_D$ be the image of compositions of $\beta_D$ and $\alpha_D$. Then $C_D$ is in $\X/w_m (K)$. Moreover, $\X/w_m (K)=\bigcup_{D} C_D \cup w_m(\X(K))$.

Since $m>427$, no $w_m$-fixed point is defined over a quadratic number field (which is equivalent to say that the class number of the order $\Z[\sqrt{-m}]$ is bigger than $2$). Therefore, $S_D \cap S_{D'}=\emptyset$ for $D\neq D'$ since any element $P$ in the intersection $S_D \cap S_{D'}$ gives rise to a $K$ rational $w_m$-fixed point. Moreover, the curve $X_0(N)/w_m$ has genus bigger than one since $m$ is big enough. Therefore there are only finitely many $D$ such that $S_D$ is non-empty. This implies that there are only finitely many $D$ such that $X^D(K)$ is non-empty.

\end{proof}

\begin{theorem}\label{thm:violation}
Assuming the conditions on $m_i$ and $d_1$ given in the beginning of the section and $N>427$,  the number of the bi-quadratic twists $\Xc$ which violate the Hasse Principle when there is no prime simultaneously ramified in $\Q(\sqrt{d_1}, \sqrt{d_2})$ and $\Q(\sqrt{-N})$, is asymptotically $M_{S_N}\frac{X}{\log^{1-\alpha}X}$. 
\end{theorem}

\begin{proof} Let $X^*(N)$ be the quotient of $X_0(N)$ by $<w_{m_1}, w_{m_2}>$. Gonzalez and Lario prove in \cite{GonzalezLario} that the genus of $X^*(N)$ is greater than $1$ when $N>159$ if $N$ is product of 2 primes $m_1$ and $m_2$. Therefore $X^*(N)$ has finitely many $K$-rational points for any number field $K$ when $N>159$. 

Let $\Xc$ be the bi-quadratic twist as defined above. In particular any element $P$ of $\Xc(\Q)$ is fixed by $\sigma_1\sigma_2w_N$. Consider $\X$ as a curve over $K=\Q(\sqrt{d_1d_2})$. Then any $P \in \Xc$ induces a $K$-rational point on the quadratic twist of $\X$ by $\zeta':\sigma_1\sigma_2 \mapsto w_N$. Let $D$ be $d_1d_2$ and $X^D(N)$ denote the quadratic twist of $\X_{/ K}$ by $\zeta'$. By Proposition \ref{fintwist}, there are only finitely many $D$ such that $X^D(K)$ is non-empty. This implies that there are only finitely many twists $\Xc$ with $\Q$-rational points since for each $D$, there are only finitely many relatively prime many tuples $(d_1,d_2)$ such that $d_1d_2=D$. Hence we have the claimed asymptotic for the number of bi-quadratic twists which has local points everywhere but no global points.
\end{proof}

Remark: It is possible to obtain generalization of Theorem \ref{thm:violation} using polyquadratic twists. In order to do this, one first needs to generalize Proposition \ref{prop:densityodd}. This is possible but we preferred to include only the biquadratic twist case for simplicity. Similarly, it is possible to generalize Theorem \ref{fintwist} by replacing $K$ with the compositum field, $\langle \Q(\sqrt{d_id_j})\rangle_ {1\leq i<j\leq k}$ and $L:=\Q(\sqrt{d_1}, \ldots, \sqrt{d_k})$. In order to make the same proof work, we need to make $m$ big enough so that the class number of $\Z[\sqrt{-m}]$ is bigger than the degree of $K$. This is possible since we know that given any $d$ there exists finitely many $\Z[\sqrt{-m}]$ with class number $d$ by the work of Deuring, Mordell and Heilbronn in 1934.

\section{Examples and Other Directions}

In this section we study the existence of global points mostly through examples. We deal with twists which has local points everywhere but no global points. At the end, we give an example of a twist which doesn't have any `'non-exceptional' rational points.

 
 Let $C$ be a smooth, projective, geometrically integral curve over $\Q$ of genus greater than or equal to $2$ with a rational degree one divisor $D$. Then we can embed $C$ into its Jacobian $J$ via the map $P\mapsto [P]-D$.
 
 Let $S$ be a finite set of primes which $C$ has good reduction at and assume that we know the generators of Mordell Weil group, $J(\Q)$. Then for every $p$ in $S$ we can compute the finite abelian group $J(\F_p)$ and the set $C(\F_p)$. Let $inj_p$ denote the injection from $C(\F_p)$ to $J(\F_p)$ and $red_p$ be the reduction map from $J(\Q)$ to $J(\F_p)$. Then we obtain the following diagram: 

$$\begin{array}{ccc}
C(\Q) & \stackrel{P \mapsto [P]-D}{\longrightarrow} & J(\Q) \\
\downarrow & & \downarrow_{red=\prod_{p \in S}red_p} \\
\prod_{p \in S} C(\F_p) & \stackrel{inj=\prod_{p \in S}inj_p}{\longrightarrow} & \prod_{p \in S} J(\F_p)\\
	
\end{array}$$

If there is a $P$ in $C(\Q)$ then $red_p([P]-D)$ is in $inj_p(C(\F_p))$ for any $p$ in $S$. In particular if images of $red$ and $inj$ do not intersect then $C(\Q)=\emptyset$.

\begin{theorem}\label{thm:sch} \emph{(\cite{sch}, \cite{sch2})} In the case of curves and under the assumption that the Tate-Shafarevich group of the jacobian of the curve is finite and the curve has a rational degree one divisor Mordell-Weil Sieve is equivalent to the Brauer-Manin obstruction.
\end{theorem}

If we suspect that $C(\Q)$ is empty then we can try to show the non-existence of global points using Mordell-Weil Sieve which is equivalent to Brauer-Manin obstruction by Theorem \ref{thm:sch}. This explicit method can be very hard to apply in practice since one needs to know equation of $C$ and generators of the Mordell Weil group of jacobian of $C$. For our family of twists, finding equations can be very tricky since we are dealing with modular curves which tend to have equations with big coefficients. Finding generators $J(\Q)$, even when the coefficients of the equation of the curve is small can be very hard as shown in \cite{Flynn} and \cite{BruinStoll}. Moreover, in order to start applying this method, we need to make sure that $\Pic^1(C)(\Q)\neq \emptyset$ which again is not easy in general even for a single curve given with an explicit equation. Moreover, we need to apply this to a family of curves most of which don't have explicit equation(at least with small coefficients). Therefore, we start with the following study about $\Pic^1(C)(\Q)$. Then we will give some explicit examples as well.


Since $C$ has local points everywhere, finding a rational degree one divisor class in enough by the following proposition:

\begin{proposition} \label{brauer}(\cite{CoMa}, Proposition 2.4, Corollary 2.5) 
If $C$ has a $\Q_p$ point then every $\Q_p$ rational divisor class contains a $\Q_p$ rational divisor. If every
$\Q_p$ rational divisor class contains a $\Q_p$ rational divisor for all primes $p$ then
every $\Q$-rational divisor class contains a $\Q$-rational divisor.

\end{proposition}


The involutions $w_m$ of $\X$ also acts on the divisor classes. We will denote by $w_m^0$ the action on $\Pic^0$ and by $w_m^1$ the action on $\Pic^1$. By definition the group of rational degree one divisor classes on $\Xc$, $$\Pic^1(\Xc)(\Q)=\{D \in \Pic^1(\Xc)(\bar{\Q}) | D^{\sigma_i}=w_{m_i}^1(D) \; \forall i\}.$$

Let $S$ be the class of $\infty-0$ in $\Pic^0(X_0(N))(\Q)$. Note that $w_N$ interchanges $0$ and $\infty$.

\begin{lemma} \label{lem:triv} 

Let $\Xc$ be a quadratic twist of $\X$ by $w_N$. There is a $D$ in $\Pic^1(\Xc)(\Q)$ if and only if there is a $P$ in $\Pic^0(X_0(N))(\K)$ such that $P^{\sigma}=w_N^0P+S$.
\end{lemma}

\begin{proof} Say there is $P \in \Pic^0(X_0(N))(\K)$ such that $P^{\sigma}=w_N^0P+S$. Let $D=P+(0)$, hence $D \in \Pic^1(X_0(N))(\K)$. Now $D^{\sigma}=P^{\sigma}+(0)^{\sigma}=P^{\sigma}+(0)$ since $(0) \in \Pic^1(X_0(N))(\Q)$ and $w_N^1D=w_N^0P+(\infty)$ hence $D^{\sigma}-w_N^1D=P^{\sigma}+(0)-w_N^0P-(\infty)=0$.

Conversely say there is $D \in \Pic^1(X_0(N))(\K)$ such that $D^{\sigma}-w_N^1D=0$. Let $P=D-(0)$. Then $P^{\sigma}=D^{\sigma}-(0)$ and $w_N^0P=w_N^1D-(\infty)$, $P^{\sigma}-w_N^0P=D^{\sigma}-w_N^1D+(\infty)-(0)=(\infty-0)$.
\end{proof}

Remark: We can extend the argument above to higher degree extensions. For instance consider the following biquadratic extension of $X_0(m_1m_2)$ by $K=\Q(\sqrt{d_1}, \sqrt{d_2})$. Let $C_1,C_2,C_3,C_4$ be the four cusps of $X_0(m_1m_2)$ such that $w_{m_1}(C_1)=C_2, w_{m_2}(C_1)=C_3$. Let $C$ denote the twisted curve, $S_1=(C_2-C_1)$ and $S_2=(C_3-C_1)$. Then $\Pic(C)(\Q)\neq \emptyset$ if and only if there exists a $P \in \Pic^0(K)$ such that $\sigma_1 P=w_{m_1}P+S_1$ and $\sigma_2 P=w_{m_2}P+S_2$. The differences of cusps are chosen so that each $w_{m_i}$ acts as $-1$ on $S_i$. For instance, let $m_1=13, m_2=2$. The cuspidal group of $X_0(26)$ is cyclic of order $21$ and the Atkin Lehner involutions act as multiplication by $8$ and $13$ on the cuspidal group. Hence $S_1=15S$ and $S_2=7S$ where $S$ is the generator of the cuspidal group. Then $P=11S$ satisfies the relations, hence any biquadratic twist of $X_0(26)$ has a degree one rational divisor class.

Manin has showed that the difference of cusps of $X_0(N)$ have finite order. Therefore we get the following corollary:

\begin{corollary}\label{cor:div} Let $\Xc$ be quadratic twist with local points everywhere. If the order of the divisor $S=(\infty-0)$ is odd in the cuspidal subgroup of $J_0(N)$ then there is a rational degree one divisor on $\Xd$.
\end{corollary}

\begin{proof} By Lemma \ref{lem:triv} there is a rational degree one divisor class on $\Xc$ if and only if there is a $P$ in $\Pic^0(X_0(N))(\K)$ such that $P^{\sigma}=w_N^0P+S$. Since $w_N$ acts as $-1$ on $J_0(N)$ and the order of $S$ is odd, there is such $P$. Then by Proposition \ref{brauer} this is equivalent to have a degree one rational divisor on $\Xc$.
\end{proof}

Since the cuspidal subgroup of $J_0(26)$ is cyclic of order $21$, by Corollary \ref{cor:div} every quadratic twist $C$ has a rational degree one divisor if $C$ has local points everywhere. Note that since the equation of $X_0(26)$ is an irreducible sextic, it is not possible to see the existence of a degee one divisor class via elementary methods.

When $N$ is prime, Ogg showed  in \cite{OggC} that the order of $S$ is given by numerator of $ \frac{N-1}{12}$. Hence we can conclude that whenever $N$ is prime and not congruent to $13$ modulo $24$, $\Pic(C)(\Q)$ is nonempty under the assumption that $C(\Q_p)$ is nonempty.

Hence this shows that the work is in the case when the order of $S$ is even. In this case we have the following result.

\begin{theorem}\label{thm:divinert} Let $N>3$ be a prime that is inert in the quadratic number field $\K$. Let $X$ be quadratic twist of $X_0(N)$ by $K$ and $w_N$. Then $\Pic^1(X)(\Q_N)$ is empty if and only if $N \equiv 1 \mod 24$ or $N \equiv 17 \mod 24$.
\end{theorem}

\begin{proof}
Say there is a $D$ in $\Pic^1(X)(\Q_N)$. Let $J_N$ be the special fiber of the Neron model of Jacobian of $X$. By Lemma \ref{lem:triv} there is a $P$ in $J_N(\F_{N^2})$ such that the difference of $\frob(P)$ and $w_N^0(P)$ is equal to the class of the divisor $(0)-(\infty)$.

By \cite{Deligne}, $J_N$ has the form $J_N^0 \times C$ where $C$ is a cyclic group generated by the class of the divisor $(0)-(\infty)$. Moreover, the group of connected components of $J_N$, which will be denoted by $\phi_N$, is isomorphic to $C$. Let $I_P$ be the connected component in which $P$ is living and $c_P$ be the element of $C$ that corresponds to $I_P$. Then the difference of $\frob(P)$ and $w_N^0(P)$ translates into $\frob(c_P)-w_N(c_P)$. Since $w_N$ interchanges $(0)$ and $(\infty)$, it acts as $-1$ on $C$. Hence the relation translates into $\frob (c_P)+(c_P)$. The map $\frob$ induces the trivial action on $C$ since the generator of $C$ is rational, hence we end up with $c_P+c_P$. However, on the other side of the equality we had $S=(0)-(\infty)$.  This gives a contradiction since the order of $S$(which is given by numerator of $\frac{N-1}{12}$, is even there cannot be an element $c_P$ in $C$ such that $c_P+c_P=S$.

Conversely if $N$ is not $1$ mod $24$, then the order of $S$ is odd hence, there is a $c$ in the cuspidal group such that $2c=S$.
\end{proof}

\textbf{Remark:} In the course of the proof of Theorem \ref{thm:divinert} it was crucial that $C$ was isomorphic to the component group of the special fiber of Neron model of $J_0(N)$. This is true for a special class of composite $N$ values as well, but not for every square-free value of $N$. Using Table 2 in \cite{Mazur2}, we see that $\phi_p$ is isomorphic to $C$ when $N=p\prod q_i$ such that there is at least one $q_i$ that is not congruent to $1$ mod $4$ and one $q_j$ such that $q_j$ not congruent to $1$ mod $3$. If we are in this case, the order of $(0)-(\infty)$ is given by $Q\frac{p-1}{12}$, where $Q=\prod (q_i+1)$. Hence, the statement of Theorem \ref{thm:divinert} can be generalized to this case as well and same proof applies.

\begin{theorem} Let $\K$ be a polyquadratic field and say $p$ is inert in $\K$, $N=p\prod q_i$ such that $p>3$ and $q_i \equiv 3 \mod 4, q_j \equiv 2 \mod 3 $ for some $i,j$. Then $\Pic^1(X)(\Q_p) = \emptyset$, where $X$ denotes the polyquadratic twist of $X_0(N)$.  
\end{theorem}

\begin{proof}
Using Table 2 in \cite{Mazur2} and the given congruence relation on $N$, we see that $\phi_p/C$ is trivial. Then like in the proof of Theorem \ref{thm:divinert}, $\Pic^1(C)(\Q_p)$ is non-empty if and only if the order of $(0)-(\infty)$ is odd. However, in this case, the order of $(0)-(\infty)$ is given by $\prod(q_i+1)\frac{p-1}{12}$ which is always an even number under the assumption of the decomposition of $N$ given in the theorem statement.
\end{proof}

In the other cases of Mazur's Table 2  in \cite{Mazur2},  $\phi_p$ is $C \oplus D$ for some abelian group $D$. Therefore we cannot apply the same method in general, without knowing how $w_N$ and $\frob$ acts on $D$. 

Once we know $\Pic^1(C)(\Q)\neq \emptyset$, and $C(\Q_p)\neq \emptyset$ for all $p$, we can try different things to study $C(\Q)$. One of these is `descent theory' (for details see for instance \cite{Stoll1}). For genus 2 curves over rationals, 2-decent has been implemented pretty efficiently in MAGMA. Using the command \texttt{TwoCoverDescent} we were able to see that in some examples $Sel^{\pi}(C)=0$ where $\pi:D \rightarrow C$ and $D$ is a two cover. When this method fails, one can try Mordell Weil Sieve or Chabuty methods. For the examples that we have encountered, if the descent failed, the rank of the jacobian was too big to apply Chabuty methods. Therefore, Mordell Weil sieve was our only choice. Note that this method requires explicit generators of $J(\Q)$ which is usually hard to compute. 

The following table summarizes the computations done for quadratic twists of $X_0(26)$. 
Note that $J_0(26)$ decomposes over $\Q$ as $E_1 \times E_2$ where $E_i$ are elliptic curves of conductor $26$. The twist of $J_0(26)$ by the cocycle $\zeta$ is isogenous to the product of twists $E_1^{\zeta} \times E_2^{\zeta}$. And since $w_{26}$ acts as multiplication by $-1$, $E_i^{\zeta}$ is the usual quadratic twist of $E_i$. Therefore the rank of $J_0^{\zeta}(\Q)$ is the sums of the ranks of $E_i^{\zeta}$. This lets us compute the exact rank of $J_0^{\zeta}(26)$. Once we know the exact rank of the jacobian then we look for generators and in the lucky cases which we were able to find the generators, we apply Mordell Weil Sieve(if the two cover descent doesn't work). 

\begin{table}[h]
     \begin{center}
       
       \caption{}
       \begin{tabular}{c|c|c|c|c|}   \hline
       $N$ & $d$ & Rank of $J$ & Descent Works? & Set of primes for MW Sieve  \\ \hline
      $26$ & $-29$ & $1$ & Yes & \\ \hline
      $26$ & $-23$ & $2$ & No & $17,31$  \\ \hline
      $26$ & $23$ & $1$ & Yes & \\ \hline
      $26$ & $29$ & $2$ & No & $5,11,23$ \\ \hline
      $26$ & $-79$ & $2$ & Yes &  \\ \hline
\end{tabular}
     \end{center}
   \end{table}

We now change the direction and give an example of a twist $S$ which has rational points. When the genus is bigger than one we know by Faltings' theorem that $C(\Q)$ is finite however there is no known algorithm that will give us all the members of this finite set. However in some cases we can find the full list. For instance, for a genus 2 curve, we can use theory of heights to come up with the generators of the Mordell Weil group of jacobian of C(which is only feasible when the height bounds are small enough) and when the rank is $1$ we can apply Chabuty methods combined with Mordell Weil sieve. This has been implemented in MAGMA. For more details see  \cite{Stoll1} and references there. Using these we get the following example:

\begin{example} The quadratic twist of $X_0(26)$ by $\Q(\sqrt{-1})$ and $w_{26}$ has only two rational points and they parametrize the elliptic curve with j-invariant 1728, hence only rational points in the twisted curve are CM points.
\end{example}

\end{document}